\documentclass{amsart}

\usepackage{amsmath,amssymb,amsthm,amsfonts, bbm}
\usepackage{url, hyperref}

\usepackage{graphicx,tikz, caption}
\usepackage{comment}

\newtheorem{theorem}{Theorem}

\newtheorem{lemma}[theorem]{Lemma}

\newtheorem{condition}[theorem]{Condition}

\theoremstyle{definition}
\newtheorem{definition}[theorem]{Definition}

\theoremstyle{remark}

\newcommand{\inner}[2]{\langle #1, #2\rangle}

\DeclareMathOperator{\im}{Im}
\DeclareMathOperator{\cone}{cone}

\begin{document}

\title[]{On the Image of Graph Distance Matrices }

\thanks{}

\author[]{William Dudarov}
\author[]{Noah Feinberg}
\author[]{Raymond Guo}
\author[]{Ansel Goh}
\author[]{Andrea Ottolini}
\author[]{Alicia Stepin}
\author[]{Raghavendra Tripathi}
\author[]{Joia Zhang}

\address{Department of Mathematics, University of Washington, Seattle, WA 98195, USA}
\email{wdudarov@uw.edu}
\email{ndarwin@uw.edu}
\email{raymondpg@outlook.com}
\email{anselgoh@uw.edu}
\email{ottolini@uw.edu}
\email{astepin@uw.edu}
\email{raghavt@uw.edu}
\email{joiaz@uw.edu}

\begin{abstract} Let $G=(V,E)$ be a finite, simple, connected, combinatorial graph on $n$ vertices and let $D \in \mathbb{R}^{n \times n}$ be its graph distance matrix $D_{ij} = d(v_i, v_j)$. Steinerberger (J. Graph Theory, 2023) empirically observed that the linear system of equations $Dx = \mathbbm{1}$, where $\mathbbm{1} = (1,1,\dots, 1)^{T}$, very frequently has a solution (even in cases where $D$ is not invertible). The smallest nontrivial example of a graph where the linear system is not solvable are two graphs on 7 vertices. We prove that, in fact, counterexamples exists for all $n\geq 7$. The construction is somewhat delicate and further suggests that such examples are perhaps rare. We also prove that for Erd\H{o}s-R\'enyi random graphs the graph distance matrix $D$ is invertible with high probability. We conclude with some structural results on the Perron-Frobenius eigenvector for a distance matrix. 
\end{abstract}


\maketitle

\section{Introduction}
 Let $G=(V,E)$ be a finite, simple, connected, combinatorial graph on $|V|=n$ vertices. A naturally associated matrix with $G$ is the \emph{graph distance matrix} $D \in \mathbb{R}^{n \times n}$ such that $D_{ij}=d(v_i, v_j)$ is the distance between the vertex $v_i$ and $v_j$.
The matrix is symmetric, integer-valued and has zero on diagonals. The graph distance matrix has been extensively studied, we refer to the survey Aouchiche-Hansen~\cite{an}. The problem of characterizing graph distance matrices was studied in~\cite{hakimi1965distance}. A result of Graham-Pollack~\cite{graham} ensures that $D$ is invertible when the graph is a tree. Invertibility of graph distance matrix continues to receive attention and various extension of Graham-Pollack has been obtained in recent times~\cite{balaji2021inverse,balaji2022inverse,hao2022inverse,hou2016inverse,zhou2017inverse,bapat2011inverse}. However, one can easily construct graphs whose distance matrices are non-invertible. Thus, in general the graph distance matrix may exhibit complex behaviour.

Our motivation comes from an observation made by Steinerberger~\cite{stein} who observed that for a graph distance matrix $D$, the linear system of equations $Dx = \mathbbm{1}$, where $\mathbbm{1}$ is a column vector of all $1$ entries, tends to frequently have a solution--even when $D$ is not invertible. An illustrative piece of statistics is as follows. Among the
\begin{align*}
& 9969 \quad &&\mbox{connected graphs in Mathematica 13.2 with}~\#V \leq 100,\\
& 3877 \quad &&\mbox{have a non-invertible distance matrix}~\mbox{rank}(D) < n~\mbox{but only} \\
& 7 \quad &&\mbox{have the property that}~\mathbbm{1} \notin \mbox{image}(D).
\end{align*}
This is certainly curious. It could be interpreted in a couple of different ways. A first natural guess would be that the graphs implemented in Mathematica are presumably more interesting than `typical' graphs and are endowed with additional symmetries. For instance, it is clear that if $D$ is the distance matrix of a vertex-transitive graph (on more than $1$ vertices) then $Dx=\mathbbm{1}$ has a solution.  Another guess would be that this is implicitly some type of statement about the equilibrium measure on finite metric spaces. For instance, it is known~\cite{stein1} that the eigenvector corresponding to the largest eigenvalue of $D$ is positive (this follows from the Perron-Frobenius theorem) and very nearly constant in the sense of all the entries having a uniform lower bound. 
The sequence A354465~\cite{OEIS} in the OEIS lists the number of graphs on $n$ vertices with $\mathbbm{1} \notin \mbox{image}(D)$ as 
$$	1, 0, 0, 0, 0, 0, 2, 14, 398, 23923, \dots$$
where the first entry corresponds to the graph on a single vertex for which $D=(0)$. 
We see that the sequence is small when compared to the number of graphs but it is hard to predict a trend based on such little information. The first nontrivial counterexamples are given by two graphs on $n=7$ vertices.

\begin{center}
    \begin{figure}[h!]
    \begin{tikzpicture}
        \node at (0,0) {\includegraphics[width=0.35\textwidth]{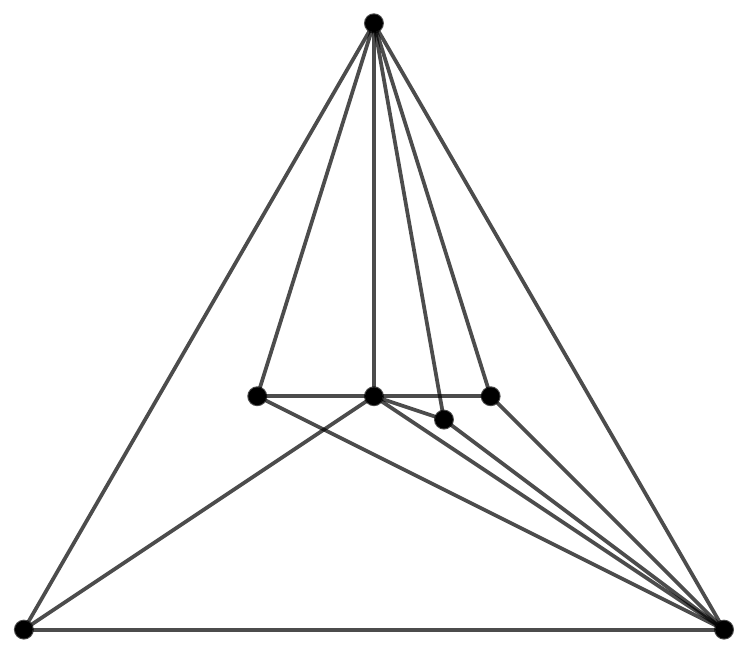}};
           \node at (5.5,0) {\includegraphics[width=0.35\textwidth]{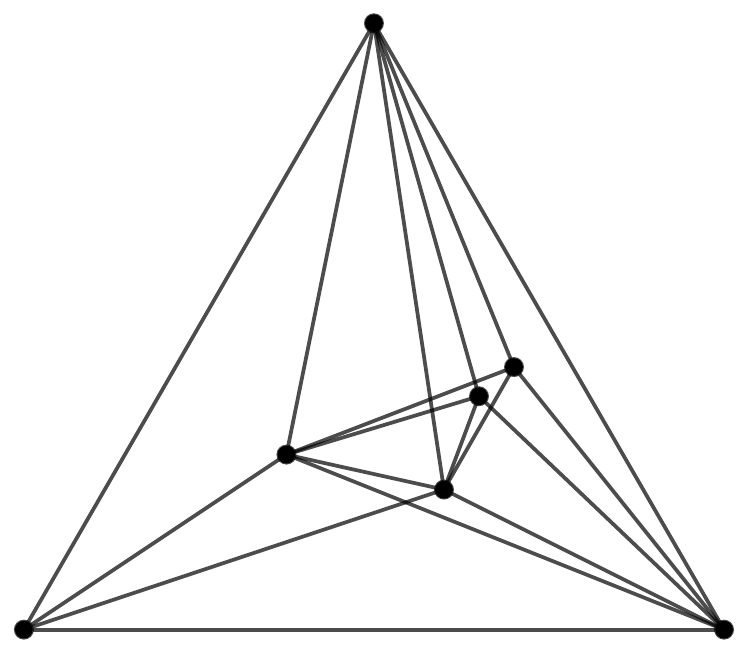}};
       \node[yshift=-2.25cm, xshift=-0.5cm] at (0,0) {(1,1,1,4)-complete 4-partite graph};
\node[yshift=-2.25cm, xshift=0.5cm] at (5.5,0) {(1,1,1,1,3)-complete k-partite graph};
    \end{tikzpicture}
    \caption{The two smallest graphs for which $\mathbbm{1} \notin \mbox{image}(D)$.}
    \label{fig:CounterExamples}
    \end{figure}
\end{center}

Lastly, it could also simply be a `small $n$' effect where the small examples behave in a way that is perhaps
not entirely representative of the asymptotic behavior. It is not inconceivable to imagine that the phenomenon disappears completely once $n$ is sufficiently large. We believe that understanding this is an interesting problem.

\subsection*{Acknowledgements}
This project was carried out under the umbrella of the Washington Experimental Mathematics Lab (WXML). The authors are grateful for useful conversations with Stefan Steinerberger. A.O. was supported by an AMS-Simons travel grant.
\section{Main Results}

\subsection{A plethora of examples}
Notice that the sequence A354465~\cite{OEIS} in the OEIS lists suggests that for $n\geq 7$ one can always find a graph on $n$ vertices for which $Dx=\mathbbm{1}$ does not have a solution. Here, we recall that $D$ represents the distance matrix of the graph, and $\mathbbm{1}$ represents a vector with all of its $|V|$ entries that are equal to one (we often omit the explicit dependence on $|V|$, when it is understood from the context). The main result of this section is the following. 
\begin{theorem}\label{thm:CE7}
    For each $n\geq 7$, there exists a graph $G$ on $n$ vertices such that $Dx=\mathbbm{1}$ does not have a solution. 
\end{theorem}
Since we know that no counterexample exists for $n<7$, the result is sharp. Our approach to find many examples of graphs for which $Dx=\mathbbm{1}$ has no solutions is to prove some structural results (of independent interests) that show how to obtain bigger examples out of smaller ones. For a careful statement of such structural results, we will need some definitions. We start with the notion of graph join. 
\begin{definition}
    The graph join $G+H$ of two graphs $G$ and $H$ is a graph on the vertex set $V(G) \cup V(H)$ with edges connecting every vertex in $G$ with every vertex in $H$ along with the edges of graph $G$ and $H$.
\end{definition}
 \begin{figure}[ht]
\includegraphics[width = 3 in]{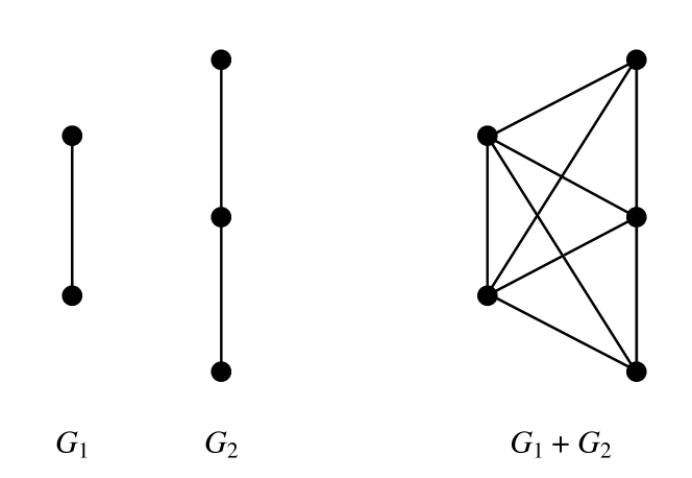}
\caption{The graph join of two paths.}
\label{graphjoin}
\end{figure}
Our structural result on the distance matrix of the graph join of two graphs is better phrased with the following definition.
\begin{definition}
    Let $G$ be a graph with adjacency matrix $A_G$. Then, define $\widetilde{D}_G = 2J-2I-A_G$. 
\end{definition}
Observe that for a graph of diameter $2$, $\widetilde{D}_G$ is the distance matrix, justifying this choice of notation. 
We now state the main ingredient in the proof of Theorem ~\ref{thm:CE7}.  
\begin{theorem}\label{thm: join}
    Let $G$ and $H$ be a graphs and suppose that $\widetilde{D}_G x=\mathbbm{1}$ has no solution. Then, the distance matrix $D$ of the graph join $G+H$  has no solution to $Dx=\mathbbm{1}$ if and only if there exists a solution to $\widetilde{D}_H x = \mathbbm{1}$ such that $\langle x, \mathbbm{1} \rangle = 0$. 
\end{theorem}

 \begin{figure}[ht]
\includegraphics[width = 4 in]{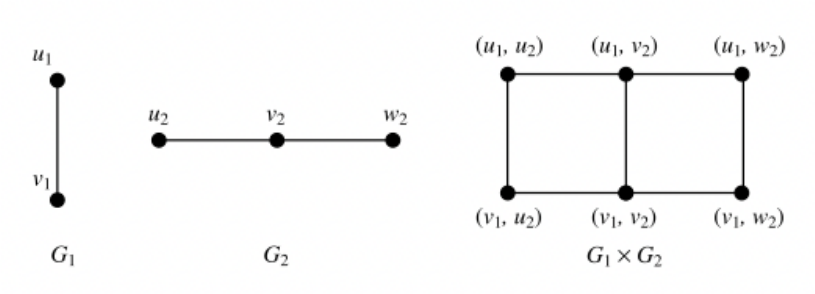}
\caption{The Cartesian product of two paths.}
\label{cartesianproduct}
\end{figure}
An alternative approach to the proof of Theorem \ref{thm:CE7}, that unfortunately does not allow for the same sharp conclusion (though it can be used to generate examples for infinitely many values of $n$) relies instead of the notion of Cartesian product.  
\begin{definition}
Given two graphs $G=(V_1, E_1)$ and $H=(V_2, E_2)$ their \emph{Cartesian product} $G \times H$ is a graph on the vertex set $V=V_1\times V_2$ such that there is an edge between vertices $(v_1,v_2)$ and $(v_1',v_2')$ if and only if either $v_1=v_1'$ and $v_2$ is adjacent to $v_2'$ in $H$ or $v_2=v_2'$ and $v_1$ is adjacent to $v_1'$ in $G$. 
\end{definition} 

\begin{theorem}\label{thm:CartesianProduct}
If $G$ and $H$ are graphs such that $\mathbbm{1}$ is not in the image of their distance matrices, then the Cartesian product graph $G \times H$ also has the property that $\mathbbm{1}$ is not in the image of its distance matrix.
\end{theorem}
We note that examples for which $Dx=\mathbbm{1}$ are not so easy to construct. In addition to the numerical evidence we provided in the introduction, we are able to give a rigorous, albeit partial, explanation of why this is the case (see Lemma \ref{lem:twos}).

\subsection{Erd\H{o}s-R\'enyi random graphs}
We conclude with a result about Erd\H{o}s-R\'enyi random graphs. We first recall their definition.
\begin{definition}
    An \textit{Erdos-Renyi} graph with parameters $(n,p)$ is a random graph on the labeled vertex set $V = \{v_1,v_2,...,v_n\}$ for which there is an edge between any pair $(v_i,v_j)$ of vertices with independent probability $p$.
\end{definition}
The following theorem shows that their distance matrices are invertible with high probability. 
As a consequence, $Dx=\mathbbm{1}$ has a solution for Erd\H{o}s-R\'enyi graphs with high probability, as we summarize in the following Theorem.

\begin{theorem}\label{thm:ErdosRenyiThm}
    Let $0 < p < 1$ and let $D_{n,p}$ be the (random) graph distance matrix associated of a random graph in $G(n,p)$. Then, as $n \rightarrow \infty$, 
    $$ \mathbb{P}\left( \det(D_{n,p}) = 0\right) \rightarrow 0.$$
\end{theorem}

It is a natural question to ask how quickly this convergence to $0$ happens. Our approach relies heavily on recent results \cite{ng} about the invertibility of a much larger class of random matrices with discrete entries, providing some explicit bounds that are likely to be loose. We propose a conjecture, which is reminiscent of work on the probability that a matrix with random $\pm 1$ Rademacher entries is singular, we refer to work of Koml{\'o}s~\cite{komlos} and the recent solution by Tikhomirov~\cite{tikh}.\\   One might be inclined to believe that the most likely way that $D_{n,p}$ can fail to be invertible is if two rows happen to be identical. This would happen if there are two vertices $v, w$ that are not connected by an edge which, for every other vertex $u \in V$, are both either connected to $u$ or not connected to $u$.
For a graph $G \in G(n,p)$ each vertex is connected to roughly $\sim np$ vertices and not connected to $\sim (1-p)n$ vertices. This motivates the following 

\begin{quote}
    \textbf{Question.} Is it true that
    $$ \lim_{n \rightarrow \infty}  \frac{ \log \left(  \mathbb{P}\left( \det(D_{n,p}) = 0\right) \right)}{n}  = \log\left( p^p (1-p)^{1-p} \right) \quad ?$$
\end{quote}

The right-hand side $\log\left( p^p (1-p)^{1-p} \right) = p \log{(p)} + (1-p) \log{(1-p)}$ is merely (up to constants) the \textit{entropy} of a Bernoulli random variable.

\subsection{Perron-Frobenius eigenvectors are nearly constant}
Let $(X, d)$ be a metric space and let $x_1, \ldots, x_n$ be $n$ distinct points in $X$. The notion of distance matrix naturally extends to this case. That is, we define $D\in \mathbb{R}^{n\times n}$ by setting $D_{ij}=d(x_i, x_j)$. This notion clearly agrees with the graph distance matrix if $X$ is a graph equipped with the usual shortest path metric. Let $\lambda_{D}$ be the Perron-Frobenious eigenvalue of $D$ and let $v$ be the corresponding eigenvector with non-negative entries. In the following we will always assume that $v$ is normalized to have $L^2$ norm $1$ unless otherwise stated. In \cite{stein1}, it was proved that
\[ \frac{\inner{v}{\mathbbm 1}}{\sqrt{n}}\geq \frac{1}{\sqrt{2}}
;.\]
It is also shown in~\cite{stein1} that the above inequality is sharp in general for the distance matrix in arbitrary metric space. However, it was observed that for graphs in the Mathematica database, the inner product tends to be very close to $1$, and it was not known if the lower bound of $1/\sqrt{2}$ is sharp for graphs. We show that this bound is sharp for graph distance matrices as well. The lower bound is achieved asymptotically by the \emph{Comet graph} that we define below. 

\begin{definition}
    We define a \emph{comet graph}, $C_{m_1}^{m_2}$, to be the disjoint union of a complete graph on $m_1$ vertices with the path graph on $m_2$ vertices and adding an edge between one end of the path graph and any vertex of the complete graph. 
    
\end{definition}
 \begin{figure}[ht]
\includegraphics[width = 3 in]{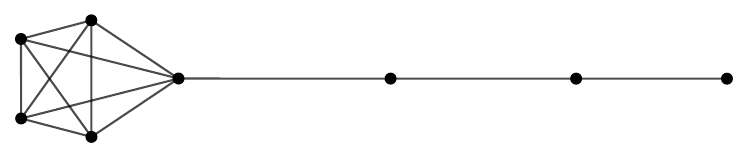}
\caption{The comet graph $C_5^3$}
\label{reduceedge}
\end{figure}

\begin{theorem}\label{thm:boundAttain}
   Let $D_m$ be the graph distance matrix of the Comet graph $C_{m^2}^m$. Let $v_m$ be the top eigenvector (normalized to have unit $L^2$ norm) of the distance matrix $D_m$. Then, 
    \[\lim_{m\to\infty} \frac{\inner{v_m}{\mathbbm 1}}{\sqrt{n}}=\frac{1}{\sqrt{2}}\;,\]
    where $n=m^2+m$ is the number of vertices in $C_{m^2}^m$. 
\end{theorem}

While Theorem~\ref{thm:boundAttain} shows that the lower bound $1/\sqrt{2}$ is sharp, it does not reveal the complete truth. It is worth emphasizing that the lower bound is achieved only in the limit as the size of the graph goes to infinity. The following theorem shows that if a graph has diameter $2$ then, $\langle v, \mathbbm{1}\rangle/\sqrt{n}$ is significantly larger. 
\begin{theorem}\label{thm:strongerBound}
    Let $G$ be a graph with diameter $2$ and let $D$ be the distance matrix of $G$. Let $v$ be the top-eigenvector of $D$ normalized to have $L^2$ norm $1$. Then,
    \[\frac{\inner{v}{\mathbbm 1}}{\sqrt{n}}\geq \frac{4}{3}\cdot \frac{1}{\sqrt{2}}\;.\]
    
\end{theorem}
In the light of above theorem, it is reasonable to expect a more general result of the following form that we leave open.  
\begin{quote}
    \textbf{Problem.} Let $G$ be a graph on $n$ vertices with distance matrix $D$. Let $v$ be the top eigenvector of $D$ with unit $L^2$ norm. If $G$ has diameter $d$ then, 
    \[\frac{\langle v, \mathbbm{1}\rangle}{n}\geq \frac{1}{\sqrt{2}}(1+f(d))\;,\]
    for some $f$ such that $f(d)\to 0$ as $d\to \infty$. 
\end{quote}

\section{Proof of Theorem~\ref{thm:CE7}}
This section is dedicated to the proof of the main Theorem~\ref{thm:CE7}. Since the main ingredient is the structural result about the distance matrix of the graph join (Theorem \ref{thm: join}), we begin the section with the proof of that. 

\begin{proof}[Proof of Theorem \ref{thm: join}]
Observe that the distance matrix of $G+H$ is given by 
    \[D = \begin{pmatrix}
\widetilde{D}_G & J \\
J & \widetilde{D}_H
\end{pmatrix}.\]
Recall that the orthogonal complement of the kernel for a symmetric matrix is the image of the matrix because the kernel of a matrix is orthogonal to the row space, which in this case, is the column space. In particular, this applies to $\widetilde{D}_G$ and $\widetilde{D}_H$.

    To prove the forwards direction, we will show the contrapositive. We have two cases, namely the case where $\widetilde{D}_H x = \mathbbm{1}$ has no solution and the case where there is a solution to $\widetilde{D}_H x = \mathbbm{1}$ where  $\langle x, \mathbbm{1} \rangle \neq 0$
    
    First, assume that $\widetilde{D}_H x = \mathbbm{1}$ has no solution. Then, we have that $\ker \widetilde{D}_G \not \perp \mathbbm{1}$ and $\ker \widetilde{D}_H \not \perp \mathbbm{1}$ because $\mathbbm{1} \not \in \im \widetilde{D}_G$ and $\mathbbm{1} \not \in \im \widetilde{D}_H$. So, there exists $x_1 \in \ker \widetilde{D}_G$ and $x_2 \in \ker \widetilde{D}_H$ such that $\langle x_1, \mathbbm{1} \rangle = \langle x_2, \mathbbm{1} \rangle = 1$. Observe that the vector $x=(x_1, x_2)^T$ satisfies $Dx=\mathbbm{1}$ so we are done with this case. 

    Now, suppose that there exists $x$ such that $\widetilde{D}_H x = \mathbbm{1}$ and $\langle x, \mathbbm{1} \rangle \neq 0$. Then, let $x_2 = x/\langle x, \mathbbm{1} \rangle$. Once again, $\ker \widetilde{D}_G \not \perp \mathbbm{1}$ so there exists $x_1 \in \ker \widetilde{D}_G$ such that $\langle x_1, \mathbbm{1} \rangle = 1-1/\langle x, \mathbbm{1} \rangle$. Then, the vector $x=(x_1, x_2)^T$ satisfies $Dx=\mathbbm{1}$. Thus, we are done with this direction. 

    Now, for the reverse direction, suppose that there exists $y$ such that $\widetilde{D}_H y= \mathbbm{1}$ and $\langle y, \mathbbm{1} \rangle = 0$. 
    Assume for a contradiction that there exists a solution to $Dx=\mathbbm{1}$. Then, we have $x_1, x_2$ such that $\widetilde{D}_G x_1 + J x_2 = \mathbbm{1}$ and 
    $J x_1 + \widetilde{D}_H x_2= \mathbbm{1}$. 

    First, suppose that $\langle x_1, \mathbbm{1} \rangle = 1$. Then, we have $\widetilde{D}_H x_2=0$ so $x_2 \in \ker \widetilde{D}_H$. Note that $\mathbbm{1} \in \im \widetilde{D}_H$ so $\ker \widetilde{D}_H \perp \mathbbm{1}$. Thus, $\langle x_2, \mathbbm{1} \rangle = 0$, implying that $Jx_2=0$. However, this implies that $\widetilde{D}_Gx_1=\mathbbm{1}$, which is a contradiction. 

    Now, suppose that $\langle x_1, \mathbbm{1} \rangle \neq 1$. Then, $\widetilde{D}_H x_2 = c\mathbbm{1}$ for some $c\neq 0$. So, $x_2= y/c + z$ for some $z \in \ker \widetilde{D}_H$. Noting that $\ker \widetilde{D}_H \perp \mathbbm{1}$, we have $\langle x_2, \mathbbm{1} \rangle = \langle y, \mathbbm{1} \rangle/c = 0$. So, $Jx_2 = 0$ implying that $\widetilde{D}_Gx_1=\mathbbm{1}$, which is a contradiction. 
\end{proof}

Now, we will construct a family of graphs $\{H_n\}_{n=3}^\infty$ such that each $H_n$ has $2n$ vertices and there exists $x$ satisfying $\widetilde{D}_{H_n}x = \mathbbm{1}$ with $\langle x, \mathbbm{1} \rangle = 1$. 
First, we will define $\{H_n\}_{i=3}^\infty$.
\begin{definition}
    For each $n \geq 3$, define $H_n=C_n^c + K_n$, where $+$ is the graph join and $C_n^c$ is the complement of the cycle graph on $n$ vertices. 
\end{definition}

\begin{lemma}\label{even}
    For each $n\geq 3$, there exists $x$ satisfying $\widetilde{D}_{H_n}x = \mathbbm{1}$ with $\langle x, \mathbbm{1} \rangle = 0$. 
\end{lemma}
\begin{proof}
To start, observe that $\widetilde{D}_{H_n}$ is of the form 
\[\begin{pmatrix}
B & J_n \\
J_n & J_n-I_n
\end{pmatrix}\]
where $B$ is defined by 
\[ B_{i,j}=\begin{cases} 
      0 & i=j \\
      2 & i=j \pm 1 \mod n\\
      1 & \text{otherwise}
   \end{cases}.
\]
The vector $x=(\mathbbm{1}_n,-\mathbbm{1}_n)^T$ satisfies $\widetilde{D}_{H_n}x=\mathbbm{1}$ with $\langle x,1 \rangle = 0$ so we are done. 
\end{proof}
Observe that each $H_i$ has an even number of vertices. We will now show construct a family of graphs $\{H_n'\}_{n=3}^\infty$ such that each $H_n'$ has $2n+1$ vertices. 

\begin{definition}
    For each $n \geq 3$, define $H_n'$ to be the graph formed by attaching one vertex to every vertex of $H_n$ except for one of the vertices of the $C_n^c$ component of $H_n$. 
\end{definition}

\begin{lemma}\label{odd}
    For each $n\geq 3$, there exists $x$ satisfying $\widetilde{D}_{H_n'}x = \mathbbm{1}$ with $\langle x, \mathbbm{1} \rangle = 0$. 
\end{lemma}
\begin{proof}
To start, observe that we can write $\widetilde{D}_{H_n'}$ as
\[\begin{pmatrix}
\widetilde{D}_{H_n}  \\
y 
\end{pmatrix}\]
where $y = (2,1, \dots, 1,0)$. 
Then, the vector $x=(\mathbbm{1}_n,-\mathbbm{1}_n,0)^T$ satisfies $\widetilde{D}_{H_n'}x=\mathbbm{1}$ with $\langle x,1 \rangle = 0$ so we are done. 
\end{proof}

Now, for sake of notation, we will recall the definition of the cone of a graph. 

\begin{definition}
    Given a graph $G$, the graph $\cone(G)$ is defined as the graph join of $G$ with the trivial graph. 
\end{definition}

\begin{proof}[Proof of Theorem~\ref{thm:CE7}]
    Take $G=\cone(H_{(n-1)/2})$ if $n$ is odd, and $G=\cone(H^{'}_{n/2-1})$ if $n$ is even. The proof is immediate from Theorem~\ref{thm: join}, Lemma~\ref{even} and Lemma~\ref{odd}.
\end{proof}

We now move to the proof of Theorem \ref{thm:CartesianProduct}, that allows for an alternative way of constructing graphs for which $Dx=\mathbbm{1}$ does not have a solution. To this aim, let $G$ and $H$ be two graphs on $n$ and $m$ vertices, respectively. Let $A\in \mathbb{R}^{n\times n}$ and $B\in \mathbb{R}^{m\times m}$ be the distance matrices of $G$ and $H$ respectively. It is well-known (see for instance~\cite[Corollary 1.35]{imrich2000product}, \cite[Lemma 1]{bap}) that the distance matrix of the Cartesian product $G\times H$ is given by $J_m \otimes A+ B \otimes J_n \in \mathbb{R}^{nm\times nm}$ where $\otimes$ is the Kronecker product and $J_{\ell}$ denotes $\ell\times \ell$ matrix with all $1$ entries. Theorem~\ref{thm:CartesianProduct} is an immediate consequence of the following Lemma~\ref{lem:productLemma}.  
\begin{lemma}\label{lem:productLemma}
Suppose that $A$ is a $n\times n$ matrix and $B$ is an $m \times m$ matrix such that the linear systems $Ay= \mathbbm{1}_{n}$ and $Bz=\mathbbm{1}_{m}$ have no solution. Then,  $$(J_m \otimes A+ B \otimes J_n)x=\mathbbm{1}_{nm}$$ has no solution. 
\end{lemma}
\begin{proof}
Assume for the sake of contradiction that there exists $x\in \mathbb{R}^{nm\times nm}$ with \[(J_m \otimes A+ B \otimes J_n)x=\mathbbm{1}_{nm}.\] Then, we have
\[ (J_m \otimes A)x = \mathbbm{1}_{nm} - (B \otimes J_n)x = (c_1, \ldots, c_m)^{T}\;,\]
where each $c_i\in \mathbb{R}^{1\times n}$ is a vector with constant entries. Since $Bz=\mathbbm{1}_{m}$ has no solutions, there must be some $1\leq j\leq m$ for which $c_{j}=\alpha\mathbbm{1}_{n}$, where $\alpha\neq 0$.
Writing $x$ as the block vector $(x_1, ..., x_m)^T$ where each $x_i\in \mathbb{R}^{1\times n}$, we note that 
\[ A(x_1+\ldots+x_m) = c_i, \quad \forall 1\leq i\leq m\;.\]
In particular the above equation holds for $i=j$. Thus, we obtain $Ay=\mathbbm{1}_{n}$ for $y= (x_1+\dots+x_m)/\alpha$ which contradicts our assumption. \end{proof}

As we pointed out in Section $2$, while we have established that there are infinitely many graphs $G$ such that $Dx=\mathbbm{1}$ does not have a solution, finding such graphs can be hard. To illustrate this, we conclude this section with a structural result about family of graphs for which $Dx=\mathbbm{1}$ does have a solution. 
\begin{lemma}\label{lem:twos}
Let $G=(V, E)$ be a connected graph. Suppose there are two vertices $v,w\in V$ such that the following conditions hold.
\begin{enumerate}
    \item $v$ is not connected to $w$
    \item $v\sim x$ for every $x\in V\setminus\{w\}$
    \item $w\sim x$ for every $x\in V\setminus\{v\}$.
\end{enumerate}
If $D$ is the graph distance matrix of $G$ then $Dx=\mathbbm{1}$ has a solution. Furthermore, if there are two or more distinct pairs of vertices satisfying $1$-$3$ then $D$ is non-invertible. 
\end{lemma}
\begin{proof}
Observe that we can write the distance of $G$ such that the first two columns of $D$ are $(0,2, 1, \dots, 1)^T$ and $(2,0, 1, \dots, 1)^T$. Therefore $x=(1/2,1/2, 0, ..., 0)^T$ satisfies $Dx=\mathbbm{1}$. 
If there are two pair of vertices, say w.l.o.g $v_1, v_2$ and $v_3, v_4$ satisfying conditions $1$-$3$ then the first four columns of $D$ look like
\[\begin{pmatrix}
0 & 2 & 1 & 1 \\
2 & 0 & 1 & 1 \\
1 & 1 & 0 & 2 \\
1 & 1 & 2 & 0 \\
1 & 1 & 1 & 1 \\
\vdots & \vdots & \vdots & \vdots \\
1 & 1 & 1 & 1
\end{pmatrix}.\]
Labeling the columns $c_1,\dots, c_4$, we have $c_1+c_2-c_3=c_4$. $D$ must be singular.
\end{proof}

\section{Proof of Theorem~\ref{thm:ErdosRenyiThm}}

We start with the following well-known result (see, e.g., \cite{klee_larman_1981}) about the diameter of an Erd\H{o}s-R\'enyi graph.




\begin{lemma}\label{diam}
    Let $p\in (0, 1)$. Let $P_{p,n}$ be the probability that a random Erd\H{o}s-R\'enyi graph $G(n, p)$ has diameter at least $3$. Then, $\lim_{n\rightarrow \infty}P_{p,n} = 0$.
\end{lemma}


Let $I$ be the identity matrix, $J$ be the all-ones matrix, and $A$ be the graph's adjacency matrix. Owing to the Lemma \eqref{diam}, we can write, with high probability, the distance matrix as $D = 2J-A-2I$.
We will now state the following theorem from~\cite{ng}, which describes the smallest singular value $\sigma_n$ of a matrix $M_n=F_n+X_n$ where $F_n$ is a fixed matrix and $X_n$ is a random symmetric matrix
 under certain conditions. 
 \begin{condition}\label{cond}
    Assume that $\xi$ has zero mean, unit variance, and there exist positive constants $c_1<c_2$ and $c_3$ such that 
    \[\mathbb{P}(c_1 \leq \lvert\xi-\xi'\rvert \leq c_2)\geq c_3,\]
    where $\xi'$ is an independent copy of $\xi$
\end{condition}
\begin{theorem}\label{big thm}
    Assume that the upper diagonal entries of $x_{ij}$ are i.i.d copies of a random variable $\xi$ satisfying \ref{cond}. Assume also that the entries $f_{ij}$ of the symmetric matrix $F_n$ satisfy $\lvert f_{ij} \rvert \leq n^\gamma$ for some $\gamma > 0$. Then, for any $B>0$, there exists $A>0$ such that 
    \[\mathbb{P}(\sigma_n(M_n)\leq n^{-A})\leq n^{-B}.\]
\end{theorem}
Combining all these results, we can prove the main result of the section.

\begin{proof}[Proof of Theorem~\ref{thm:ErdosRenyiThm}]
Owing to Lemma \ref{diam}, we can assume that with high probability the distance matrix has the form $D=2J-2A-2I$. Note that the upper diagonal entries of $A$ are i.i.d copies of a random variable satisfying Condition \ref{cond} with $c_1=c_3=1$ and $c_2=1$. Furthermore, $2(J-I)$ is symmetric and its entries are bounded. Therefore, the result follows from Theorem \ref{big thm}.

\end{proof}

\section{Proof of Theorem~\ref{thm:boundAttain}}
Let $D_m$ be the graph distance matrix of $C_{m^2}^m$. We start by observing that 
\[D_m=
\begin{bmatrix} 
    J_{m^2}-I_{m^2}&B_m\\
    (B_m)^\top& A_m
\end{bmatrix}\;,\]
where $A_m$ as a matrix $m\times m$ matrix such that $(A_m)_{ij}=|i-j|$ and $B_m$ is $m\times m$ matrix defined by
\[B_m=
\begin{bmatrix}
    2& 3 & \cdots&m+1\\
    \vdots & \vdots& \vdots & \vdots\\
    2& 3 &\cdots&m+1\\
    1& 2 &\cdots&m\\
\end{bmatrix}\]

Our first observation is that the first eigenvector of $D_m$ is constant for the first $m^2-1$ entries (considering the symmetry of the graph, this is not surprising).

\begin{lemma}
    Let $\lambda_m$ denote the largest eigenvalue of $D_m$ and let $v$ be the corresponding eigenvector. Then,    for all $i,j\leq m^2-1$, we have $v_i=v_j$.

\end{lemma}

\begin{proof}
    Let $r_i, r_j$ be $i$-th and $j$-th rows of $D$ respectively. We first note that $r_i-r_j=e_i-e_j$ for $i, j\leq m^2-1$. Now observe that
    \begin{align*}
        \lambda_m v_j-\lambda_m v_i
        &=\inner{r_j}{v}-\inner{r_i}{v}\\
        &=\inner{e_i-e_j}{v}=v_i-v_j\;.
    \end{align*}
    The conclusion follows since $\lambda_m\geq 0$.
\end{proof}

We start with an estimate for $\lambda_m$ that will later allow us to bound entries of $v$.

\begin{lemma}
    Let $\lambda_m$ be the largest eigenvalue of $D_m$ then
    $$\lambda_m = (1+o(1)) \cdot \frac{m^{5/2}}{\sqrt{3}}\;.$$
\end{lemma}

\begin{proof}
Write $D = D_m$ and let $\lambda_m$ be as above.
    Let $A$ be the $m^2 + m$ by $m^2 + m$ matrix defined by
    \begin{equation}
        A_{i,j}=
        \begin{cases}
            i - m^2 & \text{if } i > m^2, j \leq m^2\\
            j - m^2 & \text{if } j > m^2, i \leq m^2 \\
            0 & \text{otherwise}\;.
    \end{cases}
    \end{equation}   
    Let $B$ be the $m^2+m$ by $m^2+m$ matrix defined by 
    \begin{equation}
            B_{i,j}=
            \begin{cases}
                1 & \text{if } i,j \leq m^2\\
                0 & \text{otherwise} \\
        \end{cases}\;.
    \end{equation}
       Let $C$ be the $m^2+m$ by $m^2+m$ matrix defined by   
    \begin{equation}
            C_{i,j}=
            \begin{cases}
                m+1 & \text{if } i,j > m^2\\
                0 & \text{otherwise} \\
        \end{cases}\;.
    \end{equation}  
    Note that
    \[A \leq D \leq A + B + C\]
    where the inequalities refer to entrywise inequalities. This means that for all $x \in \mathbb{R}^{m^2 + m}$ with nonnegative entries,
    \[x^TAx \leq x^TDx \leq x^T(A+B+C)x\]   
    Let $\lambda_A,\lambda_B,\lambda_C$ be the top eigenvalue of $A$, $B$, and $C$ respectively and let $\lambda_{A+B+C}$ be the top eigenvalue of $A+B+C$. Noting that $A,B,C$ are all symmetric nonnegative matrices, letting $S \subset \mathbb{R}^{m^2+m}$ be the subset of vectors with nonnegative entries such that $\|x\|_2\leq 1$. Then, 
    \[\lambda_A \leq \lambda_m \leq \lambda_{A+B+C}\leq \lambda_A+\lambda_B + \lambda_C\;.\] 
    It is easily seen that $\lambda_B = m^2$ and $\lambda_C = m(m+1)$. 
    We can also compute $\lambda_A$ explicitly. Let $v$ be the top eigenvector of $A$. Since the first $m^2$ rows and columns of $M$ are all identical, the first $m^2$ entries of $v$ are the same. Normalize $v$ so that the first $m^2$ entries are 1. Then $\lambda_Av = Dv$ yields
    \[\lambda_Av_1 = \lambda_A = \sum_{j=1}^m A_{1,j}v_{m^2+j} = \sum_{j=1}^m jv_{m^2+j}\]
    and for $1 \leq k \leq m$,
    \[\lambda_Av_{m^2+k} = \sum_{j=1}^{m^2}kv_j = \sum_{j=1}^{m^2} k = m^2k\;.\]
    Plugging $v_{m^2+k} = \frac{m^2k}{\lambda_A}$ into the first equation, we get \[\lambda_A^2 = \sum_{j=1}^m m^2j^2 = \frac{m^2(m)(m+1)(2m+1)}{6}\;.\]
    This yields,
    \[\sqrt{\frac{m^3(m+1)(2m+1)}{6}} \leq \lambda_m \leq \sqrt{\frac{m^3(m+1)(2m+1)}{6}} + m^2 + m(m+1)\;.\] 
\end{proof}

With this estimate in hand we can now show stronger bounds on $\|v\|_\infty$ than are directly implied by \cite{stein1} in the general case.
\begin{lemma}
    Let $v$ be the top eigenvector of $D_m$ normalized so that $v_1=1$ we have
    \[ \| v\|_\infty = \mathcal O(\sqrt{m})\]
\end{lemma}

\begin{proof}
    It follows from~\cite{stein1} that 
    $\| v\|_\infty = \mathcal O(m)$.
    when we have normalized $v$ such that $v_1=1$.  Since the first $m^2-1$ terms of $v$ are $1$ and the entries in $D$ are at most $(m+1)$ we get
    \begin{align*}
        \lambda_m v_i
        &=\sum_{k=1}^{m^2-1} (D_m)_{i,k}v_k
        +\sum_{k=m^2}^{m^2+m} (D_m)_{i,k}v_k\\
        &\leq m^2(m+1)+2m(m+1)^2=\mathcal O(m^3)\;.
    \end{align*}
    Since $\lambda_m\geq m^{5/2}/\sqrt{3}$, it follows that $v_i\leq \mathcal{O}(\sqrt{m})$.
\end{proof}

\begin{lemma}
    Let $v$ be as above. There exists $C>0$ such that for $i\geq m^2$, we have 
    \[\sqrt{\frac{1}{3m}}-\frac{C}{m}\leq (v_i-v_{i-1})\leq \sqrt{\frac{3}{m}}+\frac{C}{m}\;,\]
    for all sufficiently large $m$. 
\end{lemma}
\begin{proof}
For $i\geq m^2$ we consider the following difference $r_i-r_{i-1}$. Observe that first $i-1$ coordinates are $1$ followed by $n+m+1-i$ many $-1$. Therefore, 
\begin{align*}
\lambda(v_i-v_{i-1})&=(D_mv)_i-(D_mv)_{i-1}=\inner{r_i-r_{i-1}}{v}\\
&=\sum_{k=1}^{i-1} v_i-\sum_{k=i}^{m^2+m} v_i
=(m^2-1)+\sum_{k=m^2}^{i-1} v_i-\sum_{k=i}^{m^2+m} v_i\;.
\end{align*}
Using the fact that $v_i\leq C\sqrt{m}$ for all $i$ we obtain 
\begin{align*}
    m^2-1-Cm^{3/2}\leq \lambda(v_i-v_{i-1})\leq m^2-1+Cm^{3/2}\;.
\end{align*}
Since $\lambda_m\sim m^{5/2}/\sqrt{3}$, the desired conclusion follows.
\end{proof}

\begin{proof}[Proof of Theorem~\ref{thm:boundAttain}]
To conclude the proof we first note that from above \[\langle \mathbbm{1}, v\rangle \geq m^2.\] On the other hand, 
 We also obtain 
 \[\|v\|_2^2 \leq 2m^2 + C(m+1)^{3/2}\;.\]
 Combining these results tells us that
 \[ \liminf_{m\to \infty}\frac{\inner{\mathbbm{1}}{v}}{\|v\|_2\cdot \|\mathbbm1\|_2}\geq \frac{1}{\sqrt{2}}\;.
 \]
\end{proof}
\section{Proof of Theorem~\ref{thm:strongerBound}}
Let $G$ be any graph with diameter $2$. Since $D_{ij}$ is either $1$ or $2$ (except for $D_{ii}=0$), it is easy to see that
    \begin{align*}
        \inner{\mathbbm 1}{v}-v_i
        \leq\lambda &v_i=\sum_{j=1}^n D_{i,j} v_j
        \leq 2(\inner{\mathbbm{1}}{v}-v_i).
    \end{align*}
    Rearranging, we obtain the uniform two-sided bound
\begin{align*}
\frac{\inner{\mathbbm 1}{v}}{\lambda+1}\leq  &v_i < 2\frac{\inner{\mathbbm{1}}{v}}{\lambda+1}.
\end{align*}
    This yields, in particular, that for all $1\leq i, j\leq n$
    \[1\leq  \frac{v_i}{v_j}\leq 2\;.\]

    This defines a convex region, that we denote by $D$. In order to prove our result, it suffices to prove that the minimum of $\|v\|_1=\inner{\mathbbm 1}{v}$ over the set $D$, subject to the constraint $\|v\|_2=1$, is at least $4/(3\sqrt{2})$. To this aim, we first notice that the minimizers of this problem are the same, up to a scalar factor, of the maximizers of $\|v_2\|_2$ in $D$ subject to $\|v\|_1=1$ (in fact, in both cases they must be minimizers of the homogeneous function $\|v\|_1/\|v\|_2$ on $D$). Since the latter is a maximization problem for a strictly convex function on a convex set, the maximizers must be extreme points of $D$. In particular, going back to the original formulation, we conclude that the smallest that
$\inner{\mathbbm 1}{v}$
    can be will be when all entries of $v$ are $c,2c$ for some $c$ so that $\|v\|_2=1$. 
Suppose now that we have $m$ entries equal to $c$ and $n-m$ equal to $2c$, then
    \begin{align*}
        1=\|v\|_2^2
        =\sum_{k=1}^m c^2+\sum_{k=m+1}^n(2c)^2
        =mc^2+(n-m)4c^2
    \end{align*}
    Then solving for $c$ we find
    \[c=\frac{1}{\sqrt{4n-3m}}\]
    So now we can optimize over $m$ to minimize the $\ell_1$ norm
    \[\frac{\|v\|_1}{\sqrt{n}}=\frac{mc+(n-m)2c}{\sqrt{n}}=\frac{2n-m}{\sqrt{n(4n-3m)}}\]
    Now treating $n$ as a constant and differentiating wrt to $m$ we get
    \begin{align*}
        \frac{d}{dm}\frac{2n-m}{\sqrt{n(4n-3m)}}
        =\frac{-\sqrt{4n^2-3mn}+\frac{3n(2n-m)}{2\sqrt{4n^2-3mn}}}{4n^2-3mn}
        =\frac{3mn-2n^2}{2(4n^2-3mn)^{\frac32}}
    \end{align*}
    If we want to set this equal to 0 we only care about the denominator so we solve
    \begin{align*}
        0&=3mn-2n^2\\
        0&=n(3m-2n)
    \end{align*}
    Which gives solutions $n=0,\frac{2n}{3}$ from which we see the latter is the minimum. Now if we substitute this into our formula for the $\ell_1$ norm we get
    \[\frac{2n-m}{\sqrt{n(4n-3m)}}=\frac{\frac{4n}{3}}{\sqrt{n(4n-2n)}}=\frac{4}{3}\cdot\frac{1}{\sqrt{2}}\]
    Now by \ref{diam} we know that if $G$ is a random graph, then for large $n$ it will have diameter 2 and this bound will hold.\\

\bibliographystyle{alpha}
\bibliography{reference}

\end{document}